\documentclass[a4paper,12pt,oneside]{amsart}       

\usepackage[british,english]{babel} 

\theoremstyle{cited}
\newtheorem{teor}{Theorem}[section]
\newtheorem{lem}[teor]{Lemma}

\newtheorem{prop}[teor]{Proposition}

\theoremstyle{definition}
\newtheorem{deft}[teor]{Definition}

\theoremstyle{remark}
\newtheorem{oss}[teor]{Remark}

\newcommand{\C}{\mathbb{C}}
\newcommand{\R}{\mathbb{R}}

\newcommand{\N}{\mathbb{N}}
\newcommand{\DD}{\boldsymbol\Delta}
\newcommand{\vol}{\textup{Vol}}

\title[Rigidity at infinity for lattices]{Rigidity at infinity for lattices in rank-one lie groups}

\author{Alessio Savini}

\begin{document}

\maketitle

\begin{abstract}
Let $\Gamma$ be a non-uniform lattice in $PU(p,1)$ without torsion and with $p\geq2 $. By following the approach developed in~\cite{franc06:articolo}, we introduce the notion of volume for a representation $\rho:\Gamma \rightarrow PU(m,1)$ where $m \geq p$. We use this notion to generalize the Mostow--Prasad rigidity theorem. More precisely, we show that given a sequence of representations $\rho_n:\Gamma \rightarrow PU(m,1)$ such that $\lim_{n \to \infty} \textup{Vol}(\rho_n) =\textup{Vol}(M)$, then there must exist a sequence of elements $g_n \in PU(m,1)$ such that the representations $g_n \circ \rho_n \circ g_n^{-1}$ converge to a reducible representation $\rho_\infty$ which preserves a totally geodesic copy of $\mathbb{H}^p_\C$ and whose $\mathbb{H}^p_\C$-component is conjugated to the standard lattice embedding $i:\Gamma \rightarrow PU(p,1) < PU(m,1)$. Additionally, we show that the same definitions and results can be adapted when $\Gamma$ is a non-uniform lattice in $PSp(p,1)$ without torsion and for representations $\rho:\Gamma \rightarrow PSp(m,1)$, still mantaining the hypothesis $m \geq p \geq 2$. 
\end{abstract}

\vspace{20pt}

%%*******************************************************
%%Materiale Iniziale
%%*******************************************************

\section{Introduction}

One of the most celebrated theorem about rigidity of lattices in semisimple Lie groups is the Mostow rigidity theorem which states that real hyperbolic lattices are strongly rigid. More precisely, let $\Gamma_1$ and $\Gamma_2$ in $PO^\circ(p,1)=\textup{Isom}^+(\mathbb{H}^p_\R)$ be two non-uniform lattices without torsion. If $\Gamma_1$ is isomorphic to $\Gamma_2$, there exists $g \in PO(p,1)$ such that $\Gamma_2=g\Gamma_1g^{-1}$. Equivalently the $p$-dimensional complete hyperbolic manifolds $M_1=\Gamma_1\backslash \mathbb{H}^p_\R$ and $M_2=\Gamma_2\backslash \mathbb{H}^p_\R$ are isometric. Here $\mathbb{H}^p_\R$ denotes the real hyperbolic space of dimension $p$.

If we assume $\Gamma < PO^\circ(p,1)$ and we look at representations $\rho:\Gamma \rightarrow PO(m,1)$, the result above may be strengthened by introducing the notion of volume for representations $\rho:\Gamma \rightarrow PO(m,1)$, where $m \geq p \geq 3$. For $m=p$ a way to define the volume of a representation $\rho$ is based on the properties of the bounded cohomology of the group $PO^\circ(p,1)$. In~\cite{bucher2:articolo} the authors prove that the
volume class $\omega_p$ is a generator for the cohomology group $H^p_{cb}(PO^\circ(p,1))$, hence, starting from it, we can construct a class in $H^p_b(\Gamma)$ by pulling back $\omega_p$ along $\rho_b^*$ and then evaluate this class with a relative fundamental class $[N,\partial N] \in H^3(N,\partial N)$ via the Kronecker
pairing. Here $N$ is any compact core of $M=\Gamma \backslash \mathbb{H}^p_\R$. 

Generalizing~\cite{bucher2:articolo}, in ~\cite{franc06:articolo} the authors define the volume for representations $\rho:\Gamma \rightarrow PO(m,1)$ when $m \geq p$ by considering the infimum all over the volumes $\textup{Vol}(f)$, where $f:\mathbb{H}^p_\R \rightarrow \mathbb{H}^m_\R$ is a properly ending smooth $\rho$-equivariant map. Intuitively a map $f$ properly ends if it respects the action of the peripheral subgroups. The existence of such a map is proved in~\cite{franc06:articolo}.

For any representation $\rho$ the volume is invariant by conjugation and hence we have a well-defined function on the character variety $X(\Gamma,PO(m,1))$ which is continuous with respect to the topology of pointwise convergence. This function is rigid, that is it satisfies $\vol(\rho)\leq \vol(M)$ and, if equality holds, $\rho$ preserves a totally geodesic copy of $\mathbb{H}^p_\R$ and its $\mathbb{H}^p_\R$-component is conjugated to the standard lattice embedding $i:\Gamma \rightarrow PO(p,1) < PO(m,1)$.

By generalizing both~\cite{culler:articolo} and~\cite{morgan:articolo}, in~\cite{morgan2:articolo} J. Morgan proposed a compactification of the variety $X(\Gamma,PO(m,1))$ whose ideal points can be interpreted as projective lenght functions of isometric $\Gamma$-actions on real trees. We call this compactification the Morgan--Shalen compactification of $X(\Gamma,PO(m,1))$. It is natural to ask if there exists a way to extend continuously the volume function to this compactification and which are the possible values attained at any ideal point. For instance, one could ask if it is possible to extend the ridigity of the volume function also at ideal points. In~\cite{savini:articolo} the authors answered to this question by proving a rigidity result which we may call strong rigidity at infinity. Let $\Gamma$ be as above and let $\rho_n:\Gamma \rightarrow PO(m,1)$ be a sequence of representations such that $\lim_{n \to \infty} \textup{Vol}(\rho_n)=\textup{Vol}(M)$. There must exist a sequence of elements $g_n \in PO(m,1)$ such that the sequence $g_n \circ \rho_n \circ g_n^{-1}$ converges to a representation $\rho_\infty$ which preserves a totally geodesic copy of $\mathbb{H}^p_\R$ and whose $\mathbb{H}^p_\R$-component is conjugated to the standard lattice embedding $i:\Gamma \rightarrow PO(p,1) < PO(m,1)$. As a consequence the sequence of representations cannot diverge to an ideal point of the character variety $X(\Gamma,PO(m,1))$, or equivalently if $\rho_n:\Gamma  \rightarrow PO(m,1)$ is a sequence of representations converging to any ideal point of the Morgan--Shalen compactification of $X(\Gamma,PO(m,1))$, then the sequence of volumes $\textup{Vol}(\rho_n)$ must be bounded from above by $\textup{Vol}(M)-\varepsilon$ with a suitable $\varepsilon>0$.

 The proof of this theorem is based essentially on the properties of the so-called BCG--natural map associated to a non-elementary representation $\rho:\Gamma \rightarrow PO(m,1)$, described in~\cite{besson95:articolo},~\cite{besson96:articolo} and~\cite{besson99:articolo}. Given such a representation there exists a map $F:\mathbb{H}^p_\R \rightarrow \mathbb{H}^m_\R$ which is equivariant with respect to $\rho$, smooth and satisfies $Jac_p(F) \leq 1$ for every $x \in \mathbb{H}^p_\R$, where $Jac_p(F)$ is the $p$-Jacobian of the map $F$. Moreover, the equality holds if and only if $D_xF$ is an isometry.

However the construction of the BCG-natural map is much more general. Let $\Gamma$ be a non-uniform lattice in $G_p$ without torsion, with either $G_p=PU(p,1)$ or $G_p=PSp(p,1)$. We say that the lattice $\Gamma$ is complex in the former case, quaternionic in the latter. Given a representation of $\rho:\Gamma \rightarrow G_m$, where $G_m=PU(m,1)$ if $\Gamma $ is complex or $G_m=PSp(m,1)$ if $\Gamma$ is quaternionic, we can adapt the procedure described by both~\cite{besson99:articolo} and~\cite{franc09:articolo} to obtain a natural map which satisfies the same properties listed previously. 

The prove the strong rigidity at infinity in both the complex and the quaternionic case we introduce the notion of volume for representations $\rho:\Gamma \rightarrow G_m$, with $m \geq p$. For uniform complex lattices the definition of volume for representations $\rho:\Gamma \rightarrow PU(m,1)$ is given both by~\cite{besson99:articolo} and by~\cite{besson07:articolo}, whereas for non-uniform complex lattices we refer to~\cite{burger4:articolo} and to~\cite{koziarz:articolo}. We find another interesting approach in~\cite{kim2:articolo}, where the authors use the pairing between bounded cohomology and $l^1$-Lipschitz homology to define the volume of a representation. However, here we give a different version of volume to adapt this notion to the non compact case, also for quaternionic lattices. By following the strategy of the proof of~\cite[Theorem 1]{savini:articolo} we get

\begin{teor}\label{convergence}
Denote by $G_p$ either $PU(p,1)$ or $PSp(p,1)$ and let $X^p$ the globally symmetric space associated to $G_p$. Let $\Gamma$ be a non-uniform lattice in $G_p$ without torsion. Assume $p \geq 2$. Let $\rho_n:\Gamma \rightarrow G_m$ be a sequence of representations with $m \geq p$. If $\lim_{n \to \infty} \textup{Vol}(\rho_n)=\textup{Vol}(\Gamma \backslash X^p)$, then there must exist a sequence of elements $g_n \in G_m$ such that the sequence $g_n \circ \rho_n \circ g_n^{-1}$ converges to a representation $\rho_\infty$ which preserves a totally geodesic copy of $X^p$ and whose $X^p$-component is conjugated to the standard lattice embedding $i:\Gamma \rightarrow G_p < G_m$.
\end{teor}

We want to stress that the quaternionic case could be obtained as a direct consequence of~\cite{corlette:articolo}, but we reported it here for sake of completeness and to show how we can get it using natural maps. 

The first section is dedicated to preliminary definitions. Let $G_p$ and $X^p$ be as in the theorem above. We briefly recall the notion of barycentre of a positive Borel measure on $\partial_\infty X^p$ and the definition of natural map $F$ associated to a non-elementary representation $\rho:\Gamma \rightarrow G_m$. We introduce the definition of volume for representations $\rho:\Gamma \rightarrow G_m$ and we compare it with the volume of the $\varepsilon$-natural maps $F^\varepsilon$. These maps are smooth, $\rho$-equavariant and converge to $F$ with respect to the $C^1$-topology, as in~\cite{franc06:articolo}. The second section is devoted to the proof of the main theorem. We conclude with some comments and remarks about the main results. 

\medskip

\textbf{Acknowledgements}: I would like to thank Marc Burger, Alessandra Iozzi and Michelle Bucher for the precious help and the enlightening conversations. I am also very grateful to Stefano Francaviglia for having introduced me to the study of BCG natural maps.  \\
I finally thank the referee for all the comments and suggestions.

%\input{materialeiniziale/Introduzione}
%%*********************************************************
%%Materiale Principale
%%*********************************************************

\section{Preliminary definitions}

% Sezione Baricentro
\subsection{Barycentre of a measure}

We start by fixing some notation. Let $G_p$ be either $PU(p,1)$ or $PSp(p,1)$. Denote by $\mathfrak{g}_p=T_eG_p$ the tangent space to $G_p$ at the neutral element. If we endow $\mathfrak{g }_p$ with its natural structure of Lie algebra, we recall that $\mathfrak{g}_p$ admits an involution $\Theta:\mathfrak{g}_p \rightarrow \mathfrak{g}_p$ which allows us to decompose $\mathfrak{g}_p=\mathfrak{l} \oplus \mathfrak{p}$, where $\mathfrak{l}$ and $\mathfrak{p}$ are the eigenspaces with respect to $1$ and $-1$ of the involution $\Theta$. Moreover $\mathfrak{p}$ is naturally identified with the tangent space at the point fixed by $\textup{exp}(\mathfrak{l})$ of the symmetric space $X^p$ associated to $G_p$ and the restriction of the Killing form to $\mathfrak{p}$ is positive definite. Via the group action we can transport this scalar product to any other tangent space of $X^p$ and hence we get in a canonical way a Riemannian metric on $X^p$. If $G_p=PU(p,1)$ then the associated symmetric space $X^p$ is the complex hyperbolic space of rank $p$, namely $\mathbb{H}^p_\C$, whereas if $G_p=PSp(p,1)$ the symmetric space $X^p$ coincides with the quaternionic hyperbolic space of rank $p$, that is $\mathbb{H}^p_\mathbb{H}$. In both cases we recall that the sectional curvature of these spaces lies between $-4$ and $-1$. In particular, since $X^p$ is negatively curved, we can talk about the visual boundary of $X^p$ and we denote it by $\partial_\infty X^p$. 

From now until the end of the paper we are going to fix a point in $X^p$ as basepoint and we are going to denote it by $O$. Moreover, we will use the same letter $O$ to denote basepoints in symmetric spaces of different dimension. Let $B_O(x,\theta)$ be the Busemann function of $X^p$ normalized at $O$, that means for every $\theta \in \partial_\infty X^p$ we set

\[
B_O(x,\theta)=\lim_{t \to \infty} d(x,c(t))-t,
\]
where $c$ is the geodesic ray starting at $O=c(0)$ and ending at $\theta=c(\infty)$.

Let $\beta$ be a positive probability measure on $\partial_\infty X^p$. Thanks to the convexity of Busemann functions the map

\[
\varphi_\beta: X^p \rightarrow \R, \hspace{10pt} \varphi_\beta(y):=\int_{\partial_\infty X^p} B_O(y,\theta)d\beta(\theta)
\]
is stricly convex, if we assume that $\beta$ is not the sum of two Dirac measures. Additionally, if the measure $\beta$ does not contain any atom of mass greater or equal than $1/2$, the following condition holds 

\[
\lim_{y \to \partial_\infty X^p} \varphi_\beta(y)=\infty.
\]
and this implies that $\varphi_\beta$ admits a unique minimum in $X^p$ (see~\cite[Appendix A]{besson95:articolo}).
On the other hand, if $\beta$ contains an atom of mass at least $1/2$, then it is easy to check that the minimum of $\varphi_\beta$ is $-\infty$ and it is attained when $y$ coincides with the atom. 

\begin{deft}
Let $\beta$ be any positive probability measure on the visual boundary $\partial_\infty X^p$ which is not the sum of two Dirac masses with the same weight. If $\beta$ contains an atom $x$ of mass greater or equal than $1/2$ then we define its \textit{barycentre} as

\[
\textup{bar}_\mathcal{B}(\beta)=x,
\]
otherwise we define it as the point

\[
\textup{bar}_\mathcal{B}(\beta)=\textup{argmin}(\varphi_\beta).
\]

The letter $\mathcal{B}$ emphasizes the dependence of the construction on the Busemann functions. The barycentre of $\beta$ will be a point in $\overline{X}^p$ which satisfies the following properties:

\begin{itemize}
	\item it is continuous with respect to the $\textup{weak-}^*$ topology on the set of probability measures on $\partial_\infty X^p$, that is if $\beta_n \to \beta$ in the
          $\textup{weak-}^*$ topology (and no measure is the sum of two atoms with equal weight)
           it holds
	
	\[
	\lim_{n \to \infty} \textup{bar}_\mathcal{B}(\beta_n)=\textup{bar}_\mathcal{B}(\beta)
	\]

	\item it is $G_p$-equivariant, indeed for every $g \in
          G_p$ (if $\beta$ is not the sum of two equal atoms) we have

	\[
	\textup{bar}_\mathcal{B}(g_*\beta)=g(\textup{bar}_\mathcal{B}(\beta)),
	\]
	
	\item when $\beta$ does not contain any atom of weight greater or equal than $1/2$, it is characterized by the following equation

          \begin{equation}\label{dadiff}
	\int_{\partial_\infty X^p} dB_O|_{(\textup{bar}_\mathcal{B}(\beta),y)}(\cdot)d\beta(y)=0.
          \end{equation}

We refer to~\cite[Appendix A]{besson95:articolo} for the equation above.

\end{itemize}

\end{deft}

% Sezione Patterson Sullivan + BCG
\subsection{The Patterson-Sullivan family of measures and the BCG--natural map}

For more details about the following definitions and constructions we recomend the reader see ~\cite{besson95:articolo}, ~\cite{besson99:articolo} and ~\cite{franc09:articolo}. Even if the last one refers only to real hyperbolic lattices, every definition that appears can be easily adapted in our context, as we will see. Before starting, fix $k=2p$ if $G_p=PU(p,1)$ and $k=4p$ if $G_p=PSp(p,1)$. The value $k$ is simply the real dimension of the symmetric space $X^p$ associated to $G_p$.

Let $\Gamma$ be a non-uniform lattice in $G_p$ without torsion, that is a discrete subgroup of $G_p$ which can be thought of as the fundamental group a complete manifold $M = X^p/\Gamma$ with finite volume and which is not compact. We say that $\Gamma$ is a complex lattice if $\Gamma < PU(p,1)$, whereas we call it quaternionic if $\Gamma < PSp(p,1)$.

\begin{deft}
The \textit{critical exponent} associated to the lattice $\Gamma$ is defined as

\[
\delta(\Gamma)=\inf \{ s \in [0,\infty]| \sum_{\gamma \in\Gamma} e^{-sd(x,\gamma x)} < \infty\}
\]
for any point $x \in X^p$. The definition above does not depend on the choice of the basepoint $x \in X^p$ used to compute the series. 
\end{deft}

When $\Gamma$ is a non-uniform lattice in $G_p$, the critical exponent is always finite and by~\cite[Theorem 2]{albuquerque:articolo} we have that $\delta(\Gamma)=k+d-2$. The number $d$ is the real dimension of the algebra on which the hyperbolic space $X^p$ is defined. Hence we have either $d=2$ if $\Gamma$ is complex or $d=4$ if $\Gamma$ is quaternionic. Moreover, we remind the reader that for $s=\delta(\Gamma)$ the series diverges by~\cite[Proposition D]{albuquerque2:articolo}, that is

\[
\sum_{\gamma \in \Gamma} e^{-\delta(\Gamma)d(x,\gamma x)}=+\infty.
\]
and for this reason we may refer to $\Gamma$ as a lattice of divergence type. 

\begin{deft}
Let $\mathcal{M}^1(Y)$ be the set of positive probability measures on a space $Y$. The \textit{family of Patterson-Sullivan measures associated to a non-uniform lattice $\Gamma$} is a family of measures \mbox{$\{\mu_x\} \in \mathcal{M}^1(\partial_\infty X^p)$}, where $x \in X^p$, which satisfies the following properties

\begin{itemize}
	\item the family is $\Gamma$-equivariant, that is $\mu_{\gamma x}=\gamma_*(\mu_x)$ for every $\gamma \in \Gamma$ and every $x \in X^p$,
	\item For every $x,y \in X^p$ it holds 

		\[
			d\mu_x(\theta)=e^{-\delta(\Gamma)B_y(x,\theta)}d\mu_y(\theta)
		\]
	where $B_y(x,\theta)$ is the Busemann function normalized at $y$.
\end{itemize}

\end{deft} 

\begin{oss}
The construction of the family of Patterson-Sullivan measures has been generalized by~\cite{albuquerque:articolo,albuquerque2:articolo} to any lattice in a Lie group $G$ of non-compact type. The support of the measures $\mu_x$ is the Furstenberg boundary $\partial_\mathcal{F} X$ of $X$, which can be thought of as the $G$-orbit of a regular point $\xi \in \partial_\infty X$. Since we are considering rank one Lie group and 

\[
\textup{codim}_{\partial_\infty X} \partial_\mathcal{F}X=\textup{rank}(X)-1
\]
we have that $\partial_\infty X=\partial_\mathcal{F}X$ in our context. 
\end{oss}

Let $\Gamma$ be as above and let $\{ \mu_x \}$ be the family of Patterson-Sullivan measures associated to
$\Gamma$. We set $\mu=\mu_O$.

Let $\rho:\Gamma \rightarrow G_m$ be a non-elementary representation. We are assuming that either $G_m=PU(m,1)$ if $\Gamma $ is complex or $G_m=PSp(m,1)$ if $\Gamma$ is quaternionic. Recall that the action of $\Gamma$ on $(\partial_\infty X^p\times \partial_\infty X^p,\mu \times \mu)$ is ergodic by~\cite{Nic89,Yue96,burger3:articolo,Rob00}, for instance.
Hence by~\cite[Corollary 3.2]{burger3:articolo} there exists a $\rho$-equivariant measurable map $$D:\partial_\infty X^p \rightarrow \partial_\infty X^m$$ and two different maps of this type must agree on a full $\mu$-measure set.

We define

\[
\nu_x:=D_*(\mu_x).
\]

Clearly the measure $\nu_x$ lives in $\mathcal{M}^1(\partial_\infty X^m)$ for every
$x$.

Since we have a non-elementary representation, $\nu_x$ does not contain any atom of mass greater or equal than $1/2$. Indeed the following holds

\begin{lem}
Let $\rho:\Gamma \rightarrow G_m$ be a non-elementary representation and let $D:\partial_\infty X^p \rightarrow \partial_\infty X^m$ be a $\rho$-equivariant measurable map. Then $D(x)\neq D(y)$ for almost every $(x,y) \in \partial_\infty X^p \times \partial_\infty X^p$.
\end{lem}

\begin{proof}
Define the set $A:=\{ (x,y) \in \partial_\infty X^p \times \partial_\infty
X^p| D(x)=D(y)\}$. Since the map $D$ is $\rho$-equivariant, $A$ is a
$\Gamma$-invariant measurable subset of $\partial_\infty X^p \times \partial_\infty
X^p$. Recall that $\Gamma$ acts ergodically on $\partial_\infty X^p \times \partial_\infty
X^p$ with respect to the measure $\mu\times\mu$. In particular, the set $A$ must have
either null measure or full measure. By contradiction, suppose that $A$ has full measure. This
implies that for almost all $x$, the slice $A(x):=\{ y
\in \partial_\infty X^p|D(x)=D(y)\}$ has full measure in $\partial_\infty
X^p$. The $G_m$-action preserves the class of $\mu$, in particular, for any $\gamma\in
\Gamma$, if $A(x)$ has full measure then so does $\gamma A(x)$.
Since $\Gamma$ is countable, this implies that for almost all $x$, the set
$A_\Gamma(x):=\cap_{\gamma \in \Gamma} \gamma^{-1} A(x)$ has full measure in $\partial_\infty
X^p$. Fix now a point $y \in A_\Gamma(x)$. For any $\gamma\in\Gamma$ we have
$(x,\gamma y)\in A$. In particular  
 
\[
D(y)=D(x)=D(\gamma y)=\rho(\gamma)D(y)
\]
for every $\gamma \in \Gamma$. We use $\gamma=id$ in the first equality and the last follows by equivariance of $D$.

The computation above would imply that $\rho$ is elementary, which is a contradiction. 
\end{proof}

By the previous lemma, for all $x \in X^p$, we can define 

\[
F(x):=\textup{bar}_\mathcal{B}(\nu_x)
\]
and this point will lie in $X^m$. In this way we get a map $F:X^p \rightarrow X^m$. 

\begin{deft}
The map $F:X^p \rightarrow X^m$ is called \textit{natural map} for the
representation $\rho:\Gamma \rightarrow G_m$.

Equation~(\ref{dadiff}) becomes

          \begin{equation}\label{dadiff2}
	\int_{\partial_\infty X^m} dB_O|_{(F(x),y)}(\cdot)d\nu_x(y)=0.
          \end{equation}
and since $\nu_x=D_*(\mu_x)$, it can be rewritten as
          \begin{equation}\label{dadiff3}
	\int_{\partial_\infty X^p} dB_O|_{(F(x),D(z))}(\cdot)d\mu_x(z)=0.
          \end{equation}

 The natural map is smooth and satisfies the following properties:

	\begin{itemize}
		\item Recall that $k$ is the real dimension of the symmetric space $X^p$. Define the $k$-Jacobian of $F$ as

\[
Jac_k(F)(x):=\max_{u_1,\ldots,u_k \in T_xX^p}||D_xF(u_1) \wedge \ldots \wedge D_xF(u_k)||_{X^m}
\]
where $\{ u_1,\ldots,u_k \}$ is an orthonormal frame of the tangent space $T_xX^p$ with respect to the standard metric induced by $g_{X^p}$ and the norm $||\cdot||_{X^m}$ is the norm on $T_{F_n(x)}X^m$ induced by $g_{X^m}$. We have $Jac_k(F)\leq 1$ and the equality holds at $x$ if and only if $D_xF:T_x X^p \rightarrow T_{F(x)}X^m$ is an isometry (see for instance~\cite[Theorem 1.10]{besson99:articolo}).
  
	\item The map $F$ is $\rho$-equivariant, that is $F(\gamma x)=\rho(\gamma)F(x)$.
	\item By differentiating~(\ref{dadiff3}), one gets that for all $x \in X^p$, $u
          \in T_x X^p$, $v \in T_{F(x)}X^m$ it holds that

	\begin{align}\label{implicit}
	&\int_{\partial_\infty X^p} \nabla dB_O|_{(F(x),D(z))}(D_xF(u),v)d\mu_x(z)=\\
	&\delta(\Gamma) \int_{\partial_\infty X^p} dB_O|_{(F(x),D(z))}(v)dB_O|_{(x,z)}(u)d\mu_x(z) \nonumber
	\end{align}

	where $\nabla$ is the Levi--Civita connection on $X^m$. 
	\end{itemize}
\end{deft}

%%Sezione Volume e Mappe approssimanti
\subsection{Volume of representations and $\varepsilon$-natural maps}\label{epsilon}

Let $\Gamma$ be a non-uniform lattice in $G_p$ without torsion. If we denote by $M=X^p/\Gamma$ we obtain a complete manifold of finite volume which is locally symmetric $X^p$ and not compact. Moreover, as a consequence of the Margulis lemma, it admits a decomposition 

\[
M:=N \cup \bigcup_{i=1}^h C_i
\]
where $N$ is a compact core of finite volume and each $C_i$ is a cuspidal neighborhood which can be seen as $N_i \times (0,\infty)$ where $\pi_1(N_i)$ is a discrete virtually nilpotent parabolic subgroup of $G_p$ (see~\cite{gromov:libro} or ~\cite{bowditch:articolo}).

As before, we set $k=2p$ if $G_p=PU(p,1)$ or $k=4p$ if $G_p=PSp(p,1)$. Let $\rho:\Gamma \rightarrow G_m$ be a representation, with $m \geq p$, and let $f:X^p \rightarrow X^m$ be a smooth $\rho$-equivariant map. By following the definition of~\cite{franc06:articolo} we want to define the volume $\textup{Vol}(f)$. Let $g_{X^m}$ be the natural metric on $X^m$. The pullback of $g_{X^m}$ along $f$ defines in a natural way a pseudo-metric on $X^p$, which can be possibly degenerate, and hence it defines a natural $k$-form given by $\tilde \omega_D=\sqrt{|\det f^*g_{X^m}|}$. The equivariance of $f$ with respect to $\rho$ implies that the form $\tilde \omega_f$ is $\Gamma$-invariant and hence it determines a $k$-form on $M$. Denote this form by $\omega_f$.

\begin{deft}
Let $\rho:\Gamma \rightarrow G_m$ be a representation and let $f:X^p \rightarrow X^m$ be any smooth $\rho$-equivariant map. The \textit{volume} of $f$ is defined as

\[
\textup{Vol}(f):=\int_M \omega_f
\]
\end{deft}

\indent We keep denoting by $f:X^p \rightarrow X^m$ a generic smooth $\rho$-equivariant map. For each cuspidal neighborhood $C_i=N_i \times (0,\infty)$, we know that $\pi_1(N_i)$ is parabolic, so it fixes a unique point in $\partial_\infty X^p$. Suppose $c_i=\textup{Fix}(\pi_1N_i)$ and let $r(t)$ be a geodesic ray ending at $c_i$. We say that $f$ is a \textit{properly ending map} if all the limit points of $f(r(t))$ lie either in $\textup{Fix}(\rho(\pi_1N_i))$ or in a finite union of $\rho(\pi_1N_i)$-invariant geodesics.   

\begin{deft}
Given a representation $\rho:\Gamma \rightarrow G_m$, we define its \textit{volume} as

\[
\textup{Vol}(\rho):=\inf \{ \textup{Vol}(f)| \text{ $f$ is smooth, $\rho$-equivariant and properly ending}\}.
\]
\end{deft}

When $\rho$ is non-elementary, a priori the BCG--natural map \mbox{$F:X^p \rightarrow
 X^m$} associated to $\rho$ is not a properly ending map, hence we cannot compare its
volume with the volume of representation $\rho$. However, by adapting the proofs contained in~\cite{franc06:articolo}, for any $\varepsilon >0$ it is possible to construct a family of smooth functions $F^\varepsilon: X^p \rightarrow
X^m$ that $C^1$-converges to $F$ as $\varepsilon \to 0$ and such that
$F^\varepsilon$ is a properly ending map for every $\varepsilon >0$.
 
\begin{deft}
For any $\varepsilon >0$ there exists a map $F^\varepsilon:X^p \rightarrow X^m$ called $\varepsilon$-\textit{natural map} associated to $\rho$ which satisfies the following properties

\begin{itemize}
	\item $F^\varepsilon$ is smooth and $\rho$-equivariant,
	\item at every point of $X^p$ we have $Jac_k(F^\varepsilon) \leq 1+\varepsilon$,
	\item for every $x \in X^p$ it holds $\lim_{\varepsilon \to 0} F^\varepsilon(x)=F(x)$ and $\lim_{\varepsilon \to 0} D_xF^\varepsilon = D_xF$,
	\item $F^\varepsilon$ is a properly ending map.
\end{itemize}
\end{deft}

The existence of $\varepsilon$-natural maps for real hyperbolic lattices is proved in~\cite[Lemma 4.5]{franc06:articolo}, but the proof can be suitably adapted both to the complex and the quaternionic case. 

\begin{oss}
The properly ending property of $F^\varepsilon$ is guaranteed by the fact that $\pi_1(N_i)$ is parabolic and stabilizes each horosphere through the fixed point $c_i$. 
In particular, since $F^\varepsilon$ is a properly ending map, it holds by definition that

\[
\textup{Vol}(\rho) \leq \int_M \sqrt{|\det((F^\varepsilon)^*g_{X^m})|}.
\]

We are going to use the previous estimate later.  
\end{oss}

The rigidity of volume holds also in this context. Indeed the same proof of~\cite{franc06:articolo} leads us to

\begin{teor}
Denote by $G_p$ either $PU(p,1)$ or $PSp(p,1)$ and let $X^p$ be the globally symmetric space associated to $G_p$. Let $\Gamma$ be a non-uniform lattice in $G_p$ without torsion and assume $p \geq 2$. Let $\rho: \Gamma \rightarrow G_m$ be a representation, where $m\geq p$. Then $\textup{Vol}(\rho) \leq \textup{Vol}(M)$ and equality holds if and only if the representation $\rho$ is a discrete and faithful representation of $\Gamma$ into the isometry group of a totally geodesic copy of $X^p$ contained in $X^m$. 
\end{teor}

For complex lattices, this result is exactly the one obtained in~\cite{burger4:articolo} or in~\cite{kim2:articolo}. The statement regarding the quaternionic case is compatible with the result obtained in~\cite{corlette:articolo}.

\section{Proof of the main theorem}

We start by fixing the following setting. 

\begin{itemize}
 
	\item A lattice $\Gamma < G_p$ where $G_p=PU(p,1)$ or $G_p=PSp(p,1)$ so that $M:=X^p/\Gamma$ is
          a (non-compact) complete manifold of finite volume. Recall that $X^p$ is the Riemannian symmetric space associated to $G_p$. Assume $p \geq 2$.
	\item A base-point $O \in X^p$ used to normalize the Busemann function $B_O(x,\theta)$, with $x \in X^p$ and $\theta \in \partial_\infty X^p$.
	\item The family $\{ \mu_x \}$ of Patterson-Sullivan probability measures associated to $\Gamma$. Set $\mu=\mu_O$.
	\item A sequence of representations $\rho_n: \Gamma \rightarrow  G_m$ such that $\lim_{n \to \infty} \textup{Vol}(\rho_n)=\textup{Vol}(M)$.
\end{itemize}

We start observing that the condition $\lim_{n \to \infty}\textup{Vol}(\rho_n) = \textup{Vol}(M)$ implies that, up to passing to a subsequence, we can suppose that no $\rho_n$ is elementary. Indeed elementary representations have zero volume and $\lim_{n \to \infty}\textup{Vol}(\rho_n) = \textup{Vol}(M)$, which is stricly positive. With an abuse of notation we still denote the subsequence of non-elementary representations by $\rho_n$. 

Since no $\rho_n$ is elementary, we can consider the sequence of $\rho_n$-equi-variant measurable maps $D_n: \partial_\infty X^p \rightarrow \partial_\infty X^m$ and the corresponding sequence of BCG--natural maps $F_n:X^p \rightarrow X^m$. Since conjugating $\rho_n$ by $g \in G_m$ reflects in post-composing $F_n$ with $g$, up to conjugating $\rho_n$ by a suitable element $g_n \in G_m$, we can suppose $F_n(O)=O$ for every $n \in \N$. 

The choice to fix the origin of $X^m$ as the image of $F_n(O)$ is made to avoid pathological behaviour. For instance consider a sequence of hyperbolic elements $g_n \in G_m$ which is divergent and define the representations $\rho_n:=g_n \circ i \circ g_n^{-1}$, where $i:\Gamma \rightarrow G_p < G_m$ is the standard lattice embedding. Clearly this sequence of representations satisfies $\lim_{n \to \infty} \textup{Vol}(\rho_n)=\textup{Vol}(M)$, since for every $n \in \N$ we have $\textup{Vol}(\rho_n)=\textup{Vol}(M)$. However, there does not exist any subsequence of $\rho_n$ converging to the holonomy of the manifold $M$. 

\begin{deft}
For any $n \in \N$ and every $x \in X^p$ we can define the following quadratic forms on $T_{F_n(x)}X^m$:

\[
k_n|_{F_n(x)}(u,u):=\langle K_n|_{F_n(x)} u, u \rangle= \int_{\partial_\infty X^p} \nabla dB_O|_{(F_n(x),D_n(\theta))}(u,u)d\mu_x(\theta)
\]

\[
h_n|_{F_n(x)}(u,u):=\langle H_n|_{F_n(x)} u, u \rangle= \int_{\partial_\infty X^p} (dB_O|_{(F_n(x),D_n(\theta))}(u))^2d\mu_x(\theta)
\]
for any $u \in T_{F_n(x)}X^m$. The notation $\langle \cdot , \cdot \rangle$ stands for the scalar product on $T_{F_n(x)}X^m$ induced by the natural metric on $X^m$. Since the rank $m$ is bigger than $p$, we will need to define another quadratic form, this time on $T_x X^p$. For any $v \in T_x X^p$, we define

\[
h'_n|_x(v,v)=\langle H'_n|_x v, v \rangle= \int_{\partial_\infty X^p} (dB_O|_{(x,\theta)}(v))^2d\mu_x(\theta).
\]
\end{deft}

For any quadratic form we are going to drop the subscript which refers to the tangent space on which the form is defined. 
Since 

\[
Jac_k(F_n)(x):=\max_{u_1,\ldots,u_k \in T_xX^p}||D_xF_n(u_1) \wedge \ldots \wedge D_xF_n(u_k)||_{X^m}, 
\]
let $\{ u_1,\ldots,u_k \}$ be any frame which realizes the maximum and denote by $U_x$ the subspace $U_x:=\textup{span}_\R\{u_1,\ldots,u_k\}$ of $T_xX^p$ (since we are working with $k$-tuples, the subspace $U_x$ coincides extacly with $T_xX^p$, but we prefer to mantain the same notation of~\cite{besson99:articolo}). Set $V_{F_n(x)}:=D_xF_n(U_x)$. We denote by $K^V_n(x)$, $H^V_n(x)$ and $H^{'U}_n(x)$ the restrictions of the form $K_n|_{F_n(x)}$, $H_n|_{F_n(x)}$ and $H_n'|_x$ to the  subspace $V_{F_n(x)}$, $V_{F_n(x)}$ and $U_x$, respectively. As consequence of the Cauchy--Schwarz inequality, as in~\cite[Section 2]{besson99:articolo} it follows that

\begin{eqnarray*}
&~&\det(K_n(x)^V)Jac_k(F_n)(x) \\&\leq&  (k+d-2)^k (\det (H^V_n(x)))^\frac{1}{2} (\det (H'^U_n(x)))^\frac{1}{2} \\&\leq&  (k+d-2)^k(\det (H^V_n(x)))^\frac{1}{2}(\textup{Tr}(H'^U_n(x))/k)^\frac{1}{2}\\&\leq& k^{-\frac{k}{2}}(k+d-2)^k(\det (H^V_n(x)))^\frac{1}{2}
\end{eqnarray*}

\begin{lem}\label{almost_everywhere}
Suppose $\lim_{n \to \infty} \textup{Vol}(\rho_n)=\textup{Vol}(M)$. Then we have that $Jac_k(F_n)$ converges to $1$ almost everywhere in $X^p$ with respect to the measure induced by the standard metric.
\end{lem}

\begin{proof}
Denote by $F_n^\varepsilon: X^p \rightarrow X^m$ the $\varepsilon$-natural maps introduced in Section~\ref{epsilon}. Recall that we have the following estimate

\[
\textup{Vol}(\rho_n) \leq \textup{Vol}(F_n^\varepsilon)=\int_M Jac_k(F^\varepsilon_n) d \textup{vol}_{X^p}(x)
\]
and since $Jac_k(F_n) \leq 1+\varepsilon$ and $\lim_{\varepsilon \to 0} D_xF^\varepsilon_n=D_xF_n$, by the dominated convergence theorem we get

\[
\textup{Vol}(\rho_n) \leq \int_M Jac_k(F_n) d\textup{vol}_{X^p}(x) \leq \textup{Vol}(M)
\]
from which it follows the statement. 
\end{proof}

By the previous lemma we have $\lim_{n \to \infty} Jac_k(F_n)=1$ almost-everywhere on $X^p$ with respect to the measure induced by the standard volume form. If $\mathcal{N}$ is the set of zero measure outside of which $Jac_k(F_n)$ converges, for every $x \in X^p \setminus \mathcal{N}$ and fixed $\varepsilon>0$ there must exist $n_0=n_0(\varepsilon,x)$ such that $Jac_k(F_n) \geq 1-\varepsilon$ for every $n > n_0$. Thus it holds

\[
\left(\frac{(k+d-2)^2}{k}\right)^\frac{k}{2}\frac{\det(H_n^V)^\frac{1}{2}}{\det(K_n^V)} > 1-\varepsilon
\]
from which we can deduce

\[
\frac{\det(H_n^V)}{(\det(K_n^V))^2} > \left(\frac{k}{(k+d-2)^2}\right)^k (1-\varepsilon)^2 > \left(\frac{k}{(k+d-2)^2}\right)^k (1-2\varepsilon).
\]

Moreover, since $X^p$ has sectional curvature which varies between $-4$ and $-1$, we  can write $K_n^V=I-H_n^V-\sum_{i=1}^{d-1}J_iH_n^VJ_i$, where $J_i(x)$ are orthogonal endomorphisms used to define the complex or the quaternionic structure on $T_{F_n(x)}X^m$ (see~\cite[Section 5.b]{besson95:articolo}). Recall that $J_i^2=-I$ at every point. Here $I$ stands for the identity on $T_{F_n(x)}X^m$. Hence, by substituting the expression of $K_n$ in the previous inequality, we get 

\begin{equation}\label{estimate}
\frac{\det(H_n^V)}{(\det(I-H_n^V-\sum_{i=1}^{d-1}J_iH_n^VJ_i))^2} > \left(\frac{k}{(k+d-2)^2}\right)^k(1-2\varepsilon).
\end{equation}

Consider now the set of positive definite symmetric matrices of dimension $k$ with real entries and trace equal to $1$, namely

\[
Sym^+_1(k,\R):=\{ H \in Sym(k,\R)| H>0, \textup{Tr}(H)=1\}.
\]

Once we have fixed a basis of $V_{F_n(x)}$, we can identify $H_n^V$, $K_n^V$ and $J_i$ with the matrices representing these endomorphisms with respect to the fixed basis. Under this assumption, recall that $H_n \in Sym^+_1(k,\R)$ for every $n \in \N$, as shown in~\cite[Proposition B.1]{besson95:articolo}. If we define 

\[
\varphi:Sym^+_1(k,\R) \rightarrow \R, \hspace{10pt} \varphi(H):=\frac{\det(H)}{(\det(I-H-\sum_{i=1}^{d-1}J_iHJ_i))^2},
\]
we know that 

\[
\varphi(H) \leq \left(\frac{k}{(k+d-2)^2}\right)^k
\]
and equality holds if and only if $H=I/k$ (see~\cite[Appendix B]{besson95:articolo}). It is worth noticing that the space $Sym^+_1(k,\R)$ is not compact and a priori there could exist a divergent sequence of elements $H_n \in Sym^+_1(k,\R)$ such that $$\lim_{n\to \infty} \varphi(H_n)=\left(\frac{k}{(k+d-2)^2}\right)^k.$$

We are going to show that this is impossible.

\begin{prop}\label{maximum}
Suppose $H_n \in Sym^+_1(k,\R)$ is sequence such that 
\[
\lim_{n\to \infty} \varphi(H_n)=\left(\frac{k}{(k+d-2)^2}\right)^k.
\]

Then the sequence $H_n$ must converge to $I/k$. 
\end{prop}

\begin{proof}
We are not going to work directly on the function $\varphi$ but we will use the auxiliary function 

\[
\psi(H):=\frac{(k-1)^\frac{2k(k-1)}{k+d-2}}{(k+d-2)^{2k}}\frac{\det(H)^\frac{k-d}{k+d-2}}{\det(I-H)^\frac{2(k-1)}{k+d-2}}.
\]

By~\cite[Lemme B.3]{besson95:articolo}, for every $H \in Sym^+_1(k,\R)$ we have that $\varphi(H) \leq \psi(H)$. Moreover both functions attain the same maximum value 

\[
\max_{H \in Sym^+_1(k,\R)} \varphi =\max_{H \in Sym^+_1(k,\R)} \psi = \left(\frac{k}{(k+d-2)^2}\right)^k
\]
at $H=I/k$. 

We are going to study the properties of the function $\psi$. We start by observing that the function $\psi$ is invariant by conjugation by an element $g \in GL(k,\R)$. Indeed, $\psi(H)$ can be expressed as 

$$\psi(H)=\frac{(k-1)^\frac{2k(k-1)}{k+d-2}}{(k+d-2)^{2k}}\frac{p_H(0)^\frac{k-d}{k+d-2}}{p_H(1)^\frac{2(k-1)}{k+d-2}},$$ 
where $p_H$ is the characteristic polynomial of $H$. Hence the claim follows. In particular, we have an induced function 

\[
\tilde \psi : O(k,\R) \backslash Sym^+_1(k,\R) \rightarrow \R, \hspace{10pt} \tilde \psi(\bar H)=\psi(H),
\]
where $\bar H$ denotes the equivalence class of the matrix $H$ and the orthogonal group $O(k,\R)$ acts on $Sym^+_1(k,\R)$ by conjugation. We can think of the space $O(k,\R) \backslash Sym^+_1(k,\R)$ as the interior $\mathring \DD_{k-1}$ of the standard $(k-1)$-simplex quotiented by the action of the symmetric group $\mathfrak{S}_k$ which permutes the coordinates of an element $(a_1,\ldots,a_k) \in \mathring \DD_{k-1}$. An explicit homeomorphism between the two spaces is given by

\[
\Lambda: O(k,\R) \backslash Sym^+_1(k,\R) \rightarrow \mathfrak{S}_k \backslash \mathring \DD_{k-1}, \hspace{10pt} \Lambda(\bar H):=[a_1(H),\ldots,a_k(H)],
\] 
where $a_i(H)$ for $i=1,\ldots,k$ are the eigenvalues of $H$. By defining $\Psi=\psi \circ \Lambda^{-1}$, we can express this function as

\[
\Psi: \mathfrak{S}_k \backslash \mathring \DD_{k-1} \rightarrow \R, \hspace{10pt} \Psi([a_1,\ldots,a_k])=\frac{(k-1)^\frac{2k(k-1)}{k+d-2}}{(k+d-2)^{2k}} \prod_{i=1}^k \frac{(a_i)^\frac{k-d}{k+d-2}}{(1-a_i)^\frac{2(k-1)}{k+d-2}}. 
\]

We are going to think of $\Psi$ as defined on $\mathring \DD_{k-1}$ and we are going to estimate this function on the boundary of $\DD_{k-1}$. Since $\sum_{i=1}^k a_i=1$, we can define the function 

\[
\hat \Psi(a_1,\ldots,a_{k-1})=\frac{(k-1)^\frac{2k(k-1)}{k+d-2}}{(k+d-2)^{2k}} \frac{(a_1 \dots a_{k-1}(1-\sum_{i=1}^{k-1}a_i))^\frac{k-d}{k+d-2}}{((1-a_1)\ldots(1-a_{k-1})(\sum_{i=1}^{k-1} a_i))^\frac{2(k-1)}{k+d-2}},
\]
which is the composition of $\Psi$ with identification of $\mathring \DD_{k-1}$ with the interior of the simplex $\tau$  in $\R^{k-1}$ whose vertices are the origin $(0,0,\ldots,0)$ and the vectors $e_i=(0,\ldots,0,1,0,\ldots,0)$ of the canonical basis, for $i=1,\ldots,k-1$. If a sequence of points is converging to a boundary point of $\DD_{k-1}$, then we have a sequence $\{ (a^{(n)}_1,\ldots,a_{k-1}^{(n)}) \}_{n \in \N}$ of points in $\tau$ converging to a boundary point. If the limit point is not a vertex of $\tau$ then $\lim_{n \to \infty} \hat \Psi(a^{(n)}_1,\ldots,a_{k-1}^{(n)})= 0$, because the numerator is converging to $0$ whereas the denominator converges to a non-zero number.

The delicate points are given by the vertices of $\tau$. On these points the function $\hat \Psi$ a priori cannot be continuously extended. Suppose we have a sequence $\{(a_1^{(n)},\ldots,a_{k-1}^{(n)})\}_{n \in \N}$ such that $\lim_{n \to \infty} (a^{(n)}_1,\ldots,a_{k-1}^{(n)})=(0,0,\ldots,0)$. We have

\[
\hat \Psi(a^{(n)}_1,\ldots,a_{k-1}^{(n)})=\frac{(k-1)^\frac{2k(k-1)}{k+d-2}}{(k+d-2)^{2k}} \frac{(a_1^{(n)} \dots a_{k-1}^{(n)}(1-\sum_{i=1}^{k-1}a_i^{(n)}))^\frac{k-d}{k+d-2}}{((1-a_1^{(n)})\ldots(1-a_{k-1}^{(n)})(\sum_{i=1}^{k-1} a_i^{(n)}))^\frac{2(k-1)}{k+d-2}},
\]  
and since we are in a neighborhood of $(0,0,\ldots,0)$ the sequence $\hat \Psi(a_1^{(n)},\ldots,a_{k-1}^{(n)})$ will have the same behaviour of the following sequence

\[
\hat \Psi(a^{(n)}_1,\ldots,a_{k-1}^{(n)}) \sim \frac{(k-1)^\frac{2k(k-1)}{k+d-2}}{(k+d-2)^{2k}} \frac{(a_1^{(n)} \dots a_{k-1}^{(n)})^\frac{k-d}{k+d-2}}{(\sum_{i=1}^{k-1} a_i^{(n)})^\frac{2(k-1)}{k+d-2}}.
\]

By looking carefully at the right-hand side of the estimate above, we can apply the inequality relating the geometric mean and arithmetic mean to get 

\begin{eqnarray*}
&~&\frac{(k-1)^\frac{2k(k-1)}{k+d-2}}{(k+d-2)^{2k}} \frac{(a_1^{(n)} \dots a_{k-1}^{(n)})^\frac{k-d}{k+d-2}}{(\sum_{i=1}^{k-1} a_i^{(n)})^\frac{2(k-1)}{k+d-2}}\\&~&\leq \frac{(k-1)^\frac{2k(k-1)}{k+d-2}}{(k+d-2)^{2k}} \frac{1}{(\sum_{i=1}^{k-1} a_i^{(n)})^\frac{2(k-1)}{k+d-2}} \left( \frac{(\sum_{i=1}^{k-1} a_i^{(n)})^{k-1}}{(k-1)^{k-1}} \right)^\frac{k-d}{k+d-2}\\&~&=\frac{(k-1)^\frac{(k-1)(k+d)}{k+d-2}}{(k+d-2)^{2k}}(\sum_{i=1}^{k-1} a_i^{(n)})^\frac{(k-1)(k-d-2)}{k+d-2}.\\
\end{eqnarray*}

The last term which appears in the inequality above depends on the exponent $k-d-2$. More precisely, by the assumption $p \geq 2$ we already know that $k \geq d+2$, but we need to distinguish the case $k=d+2$ from the case $k>d+2$. Since we assumed either $G_p=PU(p,1)$ or $G_p=PSp(p,1)$, we can have either $d=2$ or $d=4$. Thus, if $k=d+2$, we should have $k=4$ or $k=6$. The cases $k=6$ is not possible because the dimension of the tangent space of a quaternionic hyperbolic space is a multiple of $4$, so we are going to analyze only the case $k=4$. When $k=4$, the space $X^p$ becomes the complex hyperbolic space $\mathbb{H}^2_\C$ and we get the estimate

\[
\hat \Psi(0,\ldots,0) \leq \frac{3^\frac{9}{2}}{4^8}
\]
which is stricly less then the maximum of $\hat \Psi$. When $k>d+2$ the right-hand side of the inequality becomes a function which is continuous at $(0,\ldots,0)$ and it converges to $0$. Moreover, since $\Psi(0,0,\ldots,0,1)=\hat \Psi(0,\ldots,0)$ and $\Psi$ is a function which is invariant under the action of the group $\mathfrak{S}_k$, we get that

\[
\Psi(1,0,0,\ldots,0)=\Psi(0,1,0,\ldots,0)=\ldots=\Psi(0,0,\ldots,0,1)
\]
are all bounded away from the maximum value of $\Psi$. Hence, in all the possible cases, we can bound $\Psi(a_1,\ldots,a_k)$ away from its maximum in a suitable neighborhood of the boundary $\partial \DD_{k-1}$. The claim follows because $\varphi(H) \leq \psi(H)$ for every $H \in Sym^+_1(k,\R)$. 
\end{proof}

By the Estimate~(\ref{estimate}), we know that in our context we have

\[
\left(\frac{k}{(k+d-2)^2}\right)^k(1-2\varepsilon) \leq \varphi(H_n^V) \leq \left(\frac{k}{(k+d-2)^2}\right)^k
\]
for $n \geq n_0$. As a consequence of Proposition~\ref{maximum}, the sequence $H_n^V$ must converge to $I/k$. Hence $H_n^V$ converges to $I/k$ almost-everywhere on $X^p$. By following the same proof of~\cite[Proposition 3.8]{savini:articolo} we immediately get 

\begin{prop}\label{uniform}
Suppose the sequence $H_n^V$ converges almost everywhere to $I/k$. Let $K$ be any compact set of $X^p$. Then we have that

\[
\lim_{n \to \infty} ||H_n^V-\frac{I}{k}||=0
\]
uniformly for every $x \in K$. The norm which appears above is the one obtained by thinking of $H^V_n$ as an operator between Euclidean vector spaces. 
\end{prop}

As a consequence of the Cauchy--Schwarz inequality applied to Equation~(\ref{implicit}), we can write

\begin{equation} \label{cauchy}
| k_n(v,D_xF(u))| \leq (k+d-2) h_n(v,v)^{\frac{1}{2}}h_n'(u,u)^\frac{1}{2}, 
\end{equation}
for every $u \in T_xX^p$ and $v \in T_{F_n(x)}X^m$. We are going to use the previous inequality, combined with Proposition~\ref{uniform}, to prove the following

\begin{lem}\label{bound}
Let $r>0$ and let $\overline{B_r(O)}$ be the closed ball of radius $r$ around $O$ in $X^p$. Then, for $n$ large enough, there exists a real constant $C>0$ such that

\[
||D_xF_n||<C
\]
for every $x$ uniformly on $\overline{B_r(O)}$. 
\end{lem}

\begin{proof}
By Proposition~\ref{uniform}, we have that $\lim_{n \to \infty} H_n^V(x)=I/k$ for every $x$ uniformly on $\overline{B_r(O)}$. This implies that 

\[
\lim_{n \to \infty} K_n^V(x)=\frac{k+d-2}{k}I, \qquad \lim_{n \to \infty} H_n^{' U}(x)=\frac{1}{k}I.
\]
for every $x$ uniformly on $\overline{B_r(O)}$. Let $\delta >0$ and let $u \in U_x$ and $v=D_xF_n(u)$. By taking $n>n_1$, we get

\begin{align*}
	(k+d-2)/k ||D_xF_n(u)||^2_ {X^m}-\delta &\leq k_n(D_xF_n(u),D_xF_n(u))\\
	h_n(D_xF_n(u),D_xF_n(u))^\frac{1}{2} &\leq ||D_xF_n(u)||_{X^m}/\sqrt{k}+\delta\\
	h_n'(u,u)^\frac{1}{2} &\leq ||u||_{X^p}/\sqrt{k}+\delta.
\end{align*}

Hence, as a consequence of Inequality~(\ref{cauchy}), it hold that

\[
(k+d-2)/k ||D_xF_n(u)||^2_{X^m} - \delta \leq (k+d-2)(||D_xF_n(u)||_{X^m}/\sqrt{k} + \delta)(||u||_{X^p}/\sqrt{k}+\delta).
\]
By considering on both sides the supremum on all the vectors $u$ of norm equal to $1$ we get

\[
||D_xF_n||^2 < k(||D_xF_n||/\sqrt{k}+ \delta)(1/\sqrt{k}+\delta)+k\delta,
\]
hence $||D_xF_n||$ is uniformly bounded on $\overline{B_r(O)}$ for any $n>n_1$ and for any choice of $r>0$. 
\end{proof}

We are now ready to prove Theorem~\ref{convergence}.

\begin{proof}
Since we know that $\lim_{n \to \infty} \textup{Vol}(\rho_n)=\textup{Vol}(M)$, by Lemma~\ref{bound} we have that $||D_xF_n||$ must be eventually uniformly bounded $\overline{B_r(O)}$, for every $r>0$. Let $x \in X^p$ be any point and let $\gamma \in \Gamma$. Let $c$  be the geodesic joining $x$ to $\gamma x$. Denote by $L=d(x,\gamma x)$ so that the interval $[0,L]$ parametrizes the curve $c$. Consider a closed ball $\overline{B_r(O)}$ sufficiently large to contain in its interior both $x$ and $\gamma x$. On this ball there must exist a constant $C$ such that $||D_xF_n|| < C$ for $n$ bigger than a suitable value $n_0$. Thus, it holds that

\[
d(F_n(x),F_n(\gamma x)) \leq \int_0^L ||D_{c(t)}F_n(\dot c(t))|| dt \leq \int_0^L ||D_{c(t)}F_n||dt \leq Cd(x,\gamma x).
\]

Recall that given an element $g \in G_m$ its translation length is defined as $\ell_{X^m}(g):=\inf_{y \in X^m}d(gy,y)$. The previous estimate implies that the translation length of the element $\rho_n(\gamma)$ can be bounded by

\[
\ell_{X^m}(\rho_n(\gamma)) \leq d(\rho_n(\gamma)F_n(x),F_n(x)) \leq Cd(\gamma x,x).
\]

We claim that it is sufficient to bound the translation length of $\rho_n(\gamma)$ for each $\gamma \in \Gamma$ to show that the sequence $\rho_n$ is bounded in the character variety $X(\Gamma,G_m)$. Let $\gamma \in \Gamma$. If $\ell_{X^m}(\rho_n(\gamma))=0$ there is nothing to prove. Assume $\ell_{X^m}(\rho_n(\gamma))>0$. Then, up to conjugation, the element $\rho_n(\gamma)$ can be put in the form 

\[
\left(
\begin{array}{ccc}
\alpha \cosh(\ell_{X^m}(\rho_n(\gamma))/2) & \alpha \sinh(\ell_{X^m}(\rho_n(\gamma))/2) & 0\\
\alpha \sinh(\ell_{X^m}(\rho_n(\gamma))/2) & \alpha \cosh(\ell_{X^m}(\rho_n(\gamma))/2) & 0\\
0 & 0 &  M\\
\end{array}
\right)
\]
where $M=(m_{ij})$ is a matrix which lies in either in $PU(p-1)$ or in $PSp(p-1)$ and $|\alpha|=1$. Then

\[
|\textup{Tr}(\rho_n(\gamma))|\leq 2\cosh(\ell_{X^m}(\rho_n(\gamma))/2)+|\textup{Tr}(M)|\leq Cd(\gamma x,x)+C_0
\]
where $|Tr(M)|<C_0$ is obtained by a compactness argument. Hence $|\textup{Tr}(\rho_n(\gamma))|$ is bounded, as claimed. 

Moreover the choice made before to fix $F_n(O)=O$ guarantees that the sequence $\rho_n$ must converge to a representation $\rho_\infty$. Since the volume function is continuous with respect to the pointwise convergence, we get

\[
\lim_{n \to \infty} \textup{Vol}(\rho_n)=\textup{Vol}(\rho_\infty)=\textup{Vol}(M)
\] 
and hence we conclude by the rigidity of the volume function.
\end{proof}

We want to conclude with some comments about the proof. A key point to show the rigidity at infinity for both complex and quaternionic lattices is given by the estimate on the Jacobian $Jac_k(F)$ of the natural map. This estimate is sharp and this fact guarantees the rigidity at infinity of the volume function. 

In both~\cite{connell1:articolo} and~\cite{connell2:articolo} the authors generalize the construction of natural maps for lattices in Lie groups of any rank by obtaining a similar estimate on the Jacobian. The estimate is sharp for lattices in products of rank one Lie groups, but this does not hold any longer for Lie groups which are not products of this type. However, the sharpness of the estimate for lattices in products of rank one Lie groups suggests us that it should be possible to extend the strong rigidity at infinity at least in this more general context. 

%**************************
% Bibliografia
%**************************

\addcontentsline{toc}{chapter}{\bibname} 	%Inserimento bibliografia nell'indice

\vspace{20pt}
Alessio Savini\\
Department of Mathematics,\\
University of Bologna,\\
Piazza di Porta San Donato 5,\\
40126 Bologna,\\
Italy\\
alessio.savini5@unibo.it\\

\end{document}